\documentclass[a4paper,12pt,reqno]{amsart}
\usepackage{amssymb}
\usepackage{latexsym}
\usepackage{amsfonts}
\usepackage{amsmath}
\usepackage{amsthm}
\usepackage{multicol}
\theoremstyle{definition}
\newtheorem{definition}{Definition}[section]
\newtheorem{proposition}[definition]{Proposition}
\newtheorem{theorem}[definition]{Theorem}

\newtheorem{lemma}[definition]{Lemma}

\newtheorem{remark}[definition]{Remark}

\def\<{\mathop{<}}
\def\>{\mathop{>}}

\newcommand{\spt}{\mathrm{spt}\,}

\numberwithin{equation}{section}

\def\XXint#1#2#3{{\setbox0=\hbox{$#1{#2#3}{\int}$}
\vcenter{\hbox{$#2#3$}}\kern-.5\wd0}}

\title[Convergence of Landau-Lifshitz equation to multi-phase MCF]{Convergence of Landau-Lifshitz equation to multi-phase Brakke's mean curvature flow}
\author[K. Takasao]{Keisuke Takasao}
\keywords{mean curvature flow, Landau-Lifshitz equation, phase field method}
\subjclass[2010]{Primary~35K93, Secondary~53C44}
\thanks{}
\date{}
\begin{document}
\maketitle
\begin{abstract}
We study the convergence of the system of the Allen-Cahn equations to the weak solution for the multi-phase  mean curvature flow in the sense of Brakke. The Landau-Lifshitz equation in this paper can be regarded as a system of Allen-Cahn equations with the Lagrange multiplier, which is a phase field model of the multi-phase mean curvature flow. Under an assumption that the limit of the energies of the solutions for the equations matches with the total variation for the singular limit of the solutions, we show that the family of the varifolds derived from the energies is a Brakke flow.
\end{abstract}
\section{Introduction}
Let $d\geq 2$, $T>0$, $\varepsilon \in (0,1)$, and $\Omega :=\mathbb{T}^d =(\mathbb{R} /\mathbb{Z})^d$. We consider the following Landau-Lifshitz equation:
\begin{equation}
\left\{ 
\begin{array}{ll}
\varepsilon \partial _t u^{\varepsilon}  = u^{\varepsilon} \times (u^{\varepsilon}\times \nabla _{L^2} E( u^{\varepsilon} )) ,& (x,t)\in \Omega \times (0,T),  \\
u^{\varepsilon} (x,0) = u^{\varepsilon} _0 (x) ,  &x\in \Omega,
\end{array} \right.
\label{ll}
\end{equation} 
where $u^\varepsilon =(u_1 ^\varepsilon, u_2 ^\varepsilon, u_3 ^\varepsilon ): \Omega \times [0,T) \to S^2$ and $\nabla _{L^2} E( u^{\varepsilon} )$ is the $L^2$-gradient of the following energy:
\begin{equation*}
E(u^\varepsilon ) = \sum _{i=1} ^3 \int _{\Omega} \Big( \frac{\varepsilon |\nabla u_i ^\varepsilon |^2}{ 2} + \frac{W(u_i ^\varepsilon)}{\varepsilon} \Big)\, dx.
\end{equation*}
Here $\displaystyle W(s) := \frac{(1-s)^2 s^2}{2}$. Therefore we have
\begin{equation*}
\nabla _{L^2} E(u^\varepsilon) = \sum _{i=1} ^3 \Big(- \varepsilon \Delta u_i ^\varepsilon + \frac{W ' (u_i ^\varepsilon)}{\varepsilon} \Big)e_i,
\end{equation*}
where $e_1 =(1,0,0)$, $e_2 =(0,1,0)$, and $e_3 =(0,0,1)$. 

The Landau-Lifshitz equation is studied by Landau and Lifshitz~\cite{landau-lifshitz} to describe the motions of the magnetization vectors in ferromagnetic materials. In general, the inside of the ferromagnetic material is divided into several regions called magnetic domains, and the boundaries are called the domain walls~\cite{hubert-schaefer}.

Let $u^\varepsilon$ be a classical solution for \eqref{ll} and $|u ^\varepsilon (x,0) | = 1$ for any $x \in \Omega$. Then we obtain $\partial _t  u ^\varepsilon \cdot u^\varepsilon \equiv 0$ and
\begin{equation}
|u^\varepsilon (x,t)|=1 \qquad \text{for any} \ (x,t) \in \Omega \times (0,T).
\label{eq1.2}
\end{equation}
This implies that
\begin{equation}
u^{\varepsilon} \times (u^{\varepsilon}\times \nabla _{L^2} E( u^{\varepsilon} ))= -\nabla _{L^2} E(u^\varepsilon) +\lambda^\varepsilon u^\varepsilon,
\label{eq1.3}
\end{equation}
where
\begin{equation*}
\lambda^\varepsilon=\lambda^\varepsilon(x,t) =  u^\varepsilon \cdot \nabla _{L^2} E(u^\varepsilon).
\end{equation*}
Thus we obtain the following system of the Allen-Cahn equations equivalent to \eqref{ll} with $N=3$:
\begin{equation}
\left\{ 
\begin{array}{ll}
\varepsilon \partial _t u^{\varepsilon}_i  =\varepsilon \Delta u_i ^\varepsilon - \dfrac{W ' (u_i ^\varepsilon)}{\varepsilon} +\lambda ^\varepsilon u_i ^\varepsilon ,& (x,t)\in \Omega \times (0,T),  \\
u^{\varepsilon} _i (x,0) = u^{\varepsilon} _{0,i} (x) ,  &x\in \Omega,
\end{array} \right.
\label{ac}
\end{equation} 
for $i=1,2,\dots,N$. Here $u^\varepsilon =(u_1 ^\varepsilon, u_2 ^\varepsilon, \dots, u_N ^\varepsilon ): \Omega \times [0,T) \to S^2$ and
\begin{equation*}
\lambda^\varepsilon=\lambda^\varepsilon(x,t) =  u^\varepsilon \cdot \nabla _{L^2} E(u^\varepsilon)  =  \sum _{i=1} ^N  u_ i^\varepsilon \Big(- \varepsilon \Delta u_i ^\varepsilon + \frac{W ' (u_i ^\varepsilon)}{\varepsilon} \Big).
\end{equation*}
If $|u ^\varepsilon (x,0) | = 1$ for any $x \in \Omega$ then the solution for \eqref{ac} also satisfies \eqref{eq1.2}.

Note that several equations similar to \eqref{ac} have been studied as the phase field model of the multi-phase mean curvature flow~\cite{bretin-danescu-penuelas-masnou, Bretin-Denis-Lachaud-Oudet, bretin-masnou, Garcke-Nestler-Stoth1998, Garcke-Nestler-Stoth}. As an analogy of those results, it is expected that when $\varepsilon \to 0$, $\Omega$ is divided into $\Omega _t ^1$, $\Omega _t ^2$,\dots, and  $\Omega _t ^N$, and the union of the boundaries of these multi-phases is the mean curvature flow under several suitable conditions, where $\Omega_t ^i = \{ x\in \Omega \, | \, \lim _{\varepsilon \to 0} u^\varepsilon _i (x,t) =1  \}$. We rermark that it can be said that $\{ \Omega_t ^i\} _{i=1} ^N$ and  its boundaries correspond to magnetic domains and domain walls, respectively.

Recently, in \cite{bretin-danescu-penuelas-masnou, Bretin-Denis-Lachaud-Oudet}, they studied the following system of the Allen-Cahn equations:
\begin{equation}
\left\{ 
\begin{array}{ll}
\varepsilon \partial _t u^{\varepsilon}_i  =\varepsilon \Delta u_i ^\varepsilon - \dfrac{W ' (u_i ^\varepsilon)}{\varepsilon} + \Lambda ^\varepsilon \sqrt{ 2W (u_i ^\varepsilon)} ,& (x,t)\in \Omega \times (0,T),  \\
u^{\varepsilon} _i (x,0) = u^{\varepsilon} _{0,i} (x) ,  &x\in \Omega,
\end{array} \right.
\label{ac3}
\end{equation} 
where 
\begin{equation*}
\Lambda^\varepsilon = \frac{ \sum _{i=1} ^N \Big(- \varepsilon \Delta u_i ^\varepsilon + \dfrac{W ' (u_i ^\varepsilon)}{\varepsilon} \Big) }{ \sum_{i=1} ^N \sqrt{2W(u_i ^\varepsilon)} }.
\end{equation*}

The solution for \eqref{ac3} with $\sum _{i=1} ^N u_{0,i} ^\varepsilon (x)\equiv 1$ satisfies
\begin{equation}
\sum _{i=1} ^N u_i ^\varepsilon (x,t)=1 \qquad \text{for any} \ (x,t) \in \Omega \times (0,T). 
\label{eq1.7}
\end{equation}
The property \eqref{eq1.7} is similar to \eqref{eq1.2}. 
By multiplying the Lagrange multiplier $\Lambda$ by the weight function $\sqrt{2W (u_ i ^\varepsilon)}$, $\Lambda$ affects only near the boundary of $\Omega _t ^i$. Phase field methods including the Lagrange multiplier with the weight function are studied for the volume preserving mean curvature flow \cite{alfaro-alifrangis, brassel-bretin, golovaty, takasao2017}. Note that by changing the constraint \eqref{eq1.7}, the energy estimates \eqref{eq3.1} and the main results of this paper can be obtained via \eqref{ac3} (see \eqref{ac4} below). Moreover, the main results can also obtained with regard to the standard equation \eqref{ac5} for the phase field model of the multi-phase mean curvature flow.

The existence of the weak solution for the multi-phase mean curvature flow is showed by Brakke ~\cite{brakke}. However, the  weak solution called Brakke flow has a trivial solution for any initial data, that is, for any initial data $M_0$, $M_t =\emptyset$ ($t>0$) is one of the solutions for Brakke flow. 
Recently, Kim and Tonegawa~\cite{kim-tonegawa} proved the non-trivial multi-phase mean curvature flow by the method used by Brakke and Huisken's monotonicity formula~\cite{huisken1990}. Otto and Laux~\cite{laux-otto} proved the convergence of the thresholding scheme introduced by Merriman, Bence and Osher~\cite{BMO} to the weak solution for the multi-phase mean curvature flow  in the BV-framework, under the assumption similar to that in \cite{Luckhaus-Sturzenhecker}. Laux and Simon~\cite{laux-simon} also proved the convergence of the vector-valued Allen-Cahn equation without the Lagrange multiplier to the weak solution for the multi-phase mean curvature flow under similar assumption. 

The main result of this paper is that the solutions for \eqref{ac} converges to the multi-phase mean curvature flow in the  sense of Brakke under an assumption (Assumption A) similar to that in \cite{ATW, laux-otto, laux-simon, Luckhaus-Sturzenhecker}. Moreover, under the assumption, we show Huisken's monotonicity formula for \eqref{ac}.

As a related to this result, Moser~\cite{Moser2008} showed that the singular limit of the Landau-Lifshitz equation becomes a geometric flow in the sense of varifold.

The organization of the paper is as follows. In Section 2 we set out basic definitions and show the convergence theorem of the main result of this paper. In Section 3 we prove the standard energy estimates for \eqref{ac} and show the vanishing of the discrepancy measure (Lemma \ref{lem3.5} (3)) and the main results by using Assumption A below. Section 4 we mention the modified system of the Allen-Cahn equations of \eqref{ac3}.
\section{Preliminaries and main results}
For $r>0$ and $a\in \mathbb{R}^d$ we define $B_r (a):=\{ x\in \mathbb{R}^d \, | \, |x-a|<r \}$. Set $\omega _d :=\mathcal{L}^d (B_1 (0))$. We denote the space of bounded variation functions on $U\subset \mathbb{R}^d$ as $BV (U)$. For a function $f \in BV(U)$, we write the total variation measure of the distributional derivative $\nabla f$ by $\|\nabla f\|$. For a Caccioppoli set $U\subset \mathbb{R}^d$, we denote the reduced boundary of $U$ by $\partial ^\star U$. For $a=(a_1 ,a_2,\dots ,a_d)$, $b=(b_1 ,b_2,\dots ,b_d) \in \mathbb{R}^d$ we write $a\otimes b :=(a_i b_j) \in \mathbb{R}^{d\times d}$. For $A=(a_{ij}),B=(b_{ij}) \in \mathbb{R}^{d\times d}$, we define $A:B:=\sum_{i,j=1}^d a_{ij}b_{ij}$.

Next we recall several definitions from the geometric measure theory and refer to \cite{allard,brakke,federer,simon} for more details. For $k<d$, let $G(d,k)$ be the space of  $k$-dimensional subspaces of $\mathbb{R}^d$. For $S\in G (d,k)$, we also use $S$ to denote the $d$ by $d$ matrix representing the orthogonal projection $\mathbb{R}^d \to S$. 

For an open set $U\subset \mathbb{R}^d$, we denote $G_{k} (U) := U \times G(d,k)$.
We call a Radon measure on $G_k (U)$ a general $k$-varifold in $U$. We denote the set of all general $k$-varifolds by $\mathbf{V}_k(U)$. Let $V \in \mathbf{V}_k (U)$. We define a mass measure of $V$ by
\[ \| V \| (A) := V( A \times G(d,k) )  \]
for any Borel set $A \subset U$.
We also denote
\[ \| V \| (\phi ) :=  \int _{ G_k (U) } \phi (x) \, dV(x,S) \quad \text{for} \quad \phi \in C_c (\mathbb{R}^d). \]
The first variation $\delta V:C_c ^1 (U ;\mathbb{R}^d)\to \mathbb{R}$ of $V\in \mathbf{V}_k(U)$ is defined by
\[ \delta V(g):= \int _{G_k (U)} S : \nabla g(x)  \, dV(x,S) \quad \text{for} \quad g \in C_c ^1 (U ;\mathbb{R}^d).\]
We define a total variation $\| \delta V \|$ to be the largest Borel regular measure on $U$ determined by
\[ \|\delta V \| (G) := \sup \{ \delta V(g)  \, | \, g\in C_c ^1 (G;\mathbb{R}^d), \ |g|\leq 1 \} \]
for any open set $G\subset U$. If $\|\delta V \|$ is locally bounded and absolutely continuous with respect to $\|V\|$, by the Radon-Nikodym theorem, there exists a $\|V\|$-measurable function 
$ h(x)$ with values in $\mathbb{R}^d$ such that
\[  \delta V  (g) = - \int _{U } h(x)  \cdot g(x)\, d\|V\|(x)  \quad \text{for} \quad g\in C_c (U;\mathbb{R}^d). \]
We call $h$ the generalized mean curvature vector of $V$. 

We call $V \in \mathbf{V}_k (U)$ rectifiable if there exist a $\mathcal{H} ^k$ measurable countably $k$-rectifiable set $M \subset U$ and a locally $\mathcal{H}^k$ integrable function $\theta$ defined on $M$ such that
\[ V(\phi ) = \int _{M} \phi (x, T_x M) \theta (x) \, d\mathcal{H}^{k} \quad \text{for} \ \phi \in C_c (G_k (U)),\]
where $T_x M $ is the approximate tangent space of $M$ at $x$. If $\theta \in \mathbb{N}$ $\mathcal{H}^k$ a.e. on $M$, we say $V$ is integral. The sets of all rectifiable $k$-varifolds and all integral $k$-varifolds in
$U$ is denoted by $\mathbf{RV}_k (U)$ and $\mathbf{IV}_k (U)$, respectively.

\bigskip

Define
\[ \mu _t ^{i,\varepsilon} (\phi) = \sigma ^{-1} \int _{\Omega} \phi \Big( \frac{\varepsilon |\nabla u_i ^\varepsilon |^2}{ 2} + \frac{W(u_i ^\varepsilon)}{\varepsilon}  \Big) \, dx \quad \text{and} \quad \mu _t ^{\varepsilon} (\phi)= \sum _{i=1}^N \mu _t ^{i,\varepsilon} (\phi), \quad \phi \in C_c (\Omega),\]
where $\sigma = \int _0 ^1 \sqrt{2W(a)} \,da$ and $u_i ^\varepsilon$ is the solution for \eqref{ac}. Note that $\mu_t ^\varepsilon (\Omega)  = \sigma^{-1} E(u^{\varepsilon} (\cdot,t))$.

We denote the discrepancy measures $\xi _t ^{i,\varepsilon}$ and $\xi _t ^{\varepsilon}$ by
\[ \xi _t ^{i,\varepsilon} (\phi) = \sigma ^{-1} \int _{\Omega} \phi \Big( \frac{\varepsilon |\nabla u_i ^\varepsilon |^2}{ 2} - \frac{W(u_i ^\varepsilon)}{\varepsilon}  \Big) \, dx \quad \text{and} \quad \xi _t ^{\varepsilon} (\phi)= \sum _{i=1}^N \xi _t ^{i,\varepsilon} (\phi), \quad \phi \in C_c (\Omega).\]

Throughout this paper, we assume that there exists $D>0$ such that
\begin{equation}
\sup _{\varepsilon \in (0,1)} \mu_0 ^\varepsilon (\Omega) \leq D.
\label{eq2.4}
\end{equation}

To show the main result, we assume the following:

\noindent
\textbf{Assumption A} If $u_i^{\varepsilon } \to u_i $ in  $L^1 _{loc} (\Omega \times [0,T))$ for any $i=1,2,\dots, N$, then the family of the Radon measures $\{ \mu _t ^\varepsilon \}_{t \in [0,T)}$ satisfies
\begin{equation}
\lim_{\varepsilon \to 0}\int _0 ^T \mu _t ^\varepsilon (\Omega) \, dt = \int _0 ^T \sum _{i=1} ^N \| \nabla u _i (\cdot ,t) \| (\Omega) \, dt.
\label{assumption}
\tag{A}
\end{equation}
\begin{remark}
\begin{enumerate}
\item By Lemma \ref{lem3.2} and Lemma \ref{lem3.5}, there  exist a subsequence $\varepsilon \to 0$, $u_i$, and a Radon measure $\mu _t $ on $\Omega$ for any $t\in [0,T)$ such that $u_i^{\varepsilon } \to u_i $ in  $L^1 _{loc} (\Omega \times [0,T))$ for any $i=1,2,\dots, N$, $\mu _t ^\varepsilon \to \mu_t$ as Radon measures for any $t\in [0,T)$, and
\[ \mu _t (\Omega)  \geq  \sum _{i=1} ^N \| \nabla u _i (\cdot ,t) \| (\Omega) , \qquad \forall \, t \in [0,T). \]
\item Although the equations Laux and Simon~\cite{laux-simon} studied are different from those of this paper, if we replace their assumptions with the settings of this paper we have the following: $\{ \mu _t ^\varepsilon \}_{t \in [0,T)}$ satisfies 
\begin{equation}
\begin{split}
&\lim_{\varepsilon \to 0}\int _0 ^T \mu _t ^\varepsilon (\Omega) \, dt \\
=&\frac{1}{2}\int _0 ^T \sum _{1\leq i,j\leq N} ( \| \nabla u _i (\cdot ,t) \| (\Omega) + \| \nabla u_j (\cdot,t)\|(\Omega) - \| \nabla( u_i + u_j )(\cdot,t)\| (\Omega) )\, dt.
\end{split}
\label{assumption2}
\end{equation}
In the case of $N=3$, for any $t\in [0,T)$, $u_i(\cdot,t) =0$ or $1$ a.e. $x \in \Omega $ (see Lemma \ref{lem3.5} (1)) and \eqref{eq1.2} imply 
\begin{equation}
\begin{split}
\| \nabla( u_{\sigma (1)} + u_{\sigma (2)} )(\cdot,t)\| =\| \nabla u_{\sigma (3)} (\cdot,t)\|
\end{split}
\label{eq2.6}
\end{equation}
for any permutation $\sigma (\cdot)$ and $t\in [0,T)$. By \eqref{eq2.6}, \eqref{assumption2} is equivalent to \eqref{assumption}.
\end{enumerate}
\end{remark}

We define the backward heat kernel $\rho$ by
\[ \rho= \rho _{y,s} (x,t) := \frac{1}{(4\pi (s-t))^{\frac{d-1}{2}}} e^{-\frac{|x-y|^2}{4(s-t)}}, \qquad t<s, \ x,y\in \mathbb{R}^d. \]

The main result of this paper is the following:
\begin{theorem}\label{mainresults}
Let $d,N\geq 2$. Assume that $u^\varepsilon$ is a classical solution for \eqref{ac} and the initial data satisfies \eqref{eq2.4}. Then, under the Assumption A there exists a subsequence $\varepsilon \to 0$ such that the following holds:
\begin{enumerate}
\item[(a)] There exists a family of $(d-1)$-integral varifolds $\{ V_t \}_{t\in [0,T)}$ on $G_{d-1}(\Omega)$ such that
$\mu ^\varepsilon _t \to \mu _t$ as Radon measures on $\Omega $ for all $t \in [0,T)$, where $\mu_t =\| V_t \|$. 
\item[(b)] There exists $u_i \in BV _{loc}( \Omega \times [0,T)) \cap C^{\frac{1}{2}} _{loc} ([0,T) ;L^1 (\Omega))$ such that
\begin{enumerate}
\item[(b1)] $u_i ^{\varepsilon} \to u_i \ \ \text{in} \ L^1 _{loc} ( \Omega \times [0,T))$ and a.e. pointwise.
\item[(b2)] $u_i =0 $ or $1$ a.e. on $\Omega \times [0,T)$, and $\sum _{i=1} ^N u_i =1$ a.e. on $\Omega \times [0,T)$.
\item[(b3)] $\| \nabla u_i (\cdot,t) \| (\phi) \leq \mu _t ^i (\phi) $ for any $t\in [0,T)$ and $\phi \in C_c (\Omega;\mathbb{R}^+)$. Moreover $\| \nabla u_i (\cdot,t) \| (\Omega) = \mu _t ^i (\Omega) $ and $\lim _{\varepsilon \to 0} \xi _t ^\varepsilon (\Omega)=0$ a.e. $t\in [0,T)$.
\item[(b4)] $V_t \in \mathbf{IV} _{d-1} (\Omega)$ a.e. $t\in [0,T)$.
\end{enumerate}
\item[(c)] $\{ V _t \}_{t\in (0,T)}$ is the Brakke flow, that is, $\{ V _t \}_{t\in (0,T)}$ satisfies 
\begin{equation}
\int _{\Omega} \phi \, d\| V_t \| \Big|_{t=t_1} ^{t_2} \leq \int_{t_1} ^{t_2} \int _{\Omega} \Big(- \phi |h|^2 + \nabla \phi \cdot h +\partial _t \phi\Big) \, d\| V_t \| dt
\label{brakke}
\end{equation}
for any $\phi \in C_c ^1 (\Omega\times [0,T) )$ and $0\leq  t_1 <t_2 <T$.
Here $h$ is the generalized mean curvature vector of $V_t$. 
\item[(d)] (Huisken's monotonicity formula) We have
\begin{equation}
\begin{split}
\frac{d}{dt} \int _{\mathbb{R}^d} \rho_{y,s} (x,t) \, d\mu _t (x) \leq 0, 
\end{split}
\label{monoton}
\end{equation}
for any $x,y \in \mathbb{R}^d$ and $t<s$ with $t\in [0,T)$. Here $\mu _t$ is extended periodically to $\mathbb{R}^d$.
\end{enumerate}
\end{theorem}
\begin{remark}
Without Assumption A, we have
\begin{equation}
\begin{split}
\frac{d}{dt} \int _{\mathbb{R}^d} \rho \, d\mu^\varepsilon _t (x) \leq \frac{1}{2(s-t)}\int _{\mathbb{R}^d}  \rho \, d\xi^\varepsilon _t (x)
\end{split}
\label{monoton3}
\end{equation}
for any $\varepsilon \in (0,1)$, $x,y \in \mathbb{R}^d$ and $t<s$ with $t\in [0,T)$ (see \eqref{monoton2} below). Thus, if there exists a suitable estimate for the function $\sum_{i=1} ^N \Big( \frac{\varepsilon |\nabla u_i ^\varepsilon |^2}{ 2} - \frac{W(u_i ^\varepsilon)}{\varepsilon} \Big)$ such as the negativity, then by  an argument similar to that in \cite{Ilmanen}, we can expect to obtain the vanishing of $\xi_t ^\varepsilon$ and the multi-phase Brakke flow via the phase field method, without Assumption A.
\end{remark}
\section{Proofs}
First we show the standard energy estimates for \eqref{ac}. In this section we assume that $u^\varepsilon$ is the classical solution for \eqref{ac} with \eqref{eq2.4}. 

\bigskip

By \eqref{ac} we have
\begin{equation*}
\varepsilon |\partial _t u^\varepsilon |^2 =  -\nabla _{L^2} E(u^\varepsilon) \cdot \partial _t u^\varepsilon + \lambda^\varepsilon u^\varepsilon\cdot \partial _t u^\varepsilon=-\nabla _{L^2} E(u^\varepsilon) \cdot \partial _t u^\varepsilon
\end{equation*}
and
\begin{equation*}
\varepsilon \partial _t u^\varepsilon \cdot \partial _{x_k} u^\varepsilon  =  -\nabla _{L^2} E(u^\varepsilon) \cdot \partial _{x_k} u^\varepsilon + \lambda^\varepsilon u^\varepsilon\cdot \partial _{x_k} u^\varepsilon=-\nabla _{L^2} E(u^\varepsilon) \cdot \partial _{x_k} u^\varepsilon,
\end{equation*}

where \eqref{eq1.2} is used. Then we obtain
\begin{proposition}\label{prop3.1}
\begin{equation}
\frac{d}{dt} \mu_t ^\varepsilon (\Omega) = -\sigma ^{-1}\int _{\Omega} \varepsilon |\partial _t u ^\varepsilon|^2 \, dx , \quad \mu_T ^\varepsilon (\Omega)  + \sigma ^{-1}\int _0 ^T \int _{\Omega} \varepsilon | \partial _t  u ^\varepsilon|^2 \, dx dt=\mu_0 ^\varepsilon (\Omega) \leq D 
\label{eq3.1}
\end{equation}
and
\begin{equation}
\sum_{i=1} ^N \int _{\Omega} \varepsilon \partial _t u^\varepsilon _i \nabla u^\varepsilon _i \cdot g \, dx  
= \sum_{i=1} ^N \int _{\Omega} \Big( \varepsilon \Delta u_i ^\varepsilon - \frac{W ' (u_i ^\varepsilon)}{\varepsilon}\Big) \nabla u_i ^\varepsilon \cdot g \, dx, \quad g \in (C(\Omega))^d.
\label{eq3.2}
\end{equation}
\end{proposition}

\bigskip

By \eqref{eq3.1}  we have
\begin{lemma}\label{lem3.2}
There exist a family of  a Radon measures $\{ \mu _t \}_{t \in [0,T)} $ on $\Omega$ and a subsequence $\{ \varepsilon_j \}_{j=1} ^\infty $ (denoted by the same index) such that
\begin{equation}
\mu_t ^{\varepsilon _j}\to \mu_t \qquad \text{as Radon measures for any} \ t\in [0,T). 
\label{eq3.3}
\end{equation}
\end{lemma}
The proof is similar to that in \cite[Proposition 4.2]{mugnai-roger2008}. So we skip the proof.
\begin{definition} If $ |\nabla u _i ^\varepsilon |\not=0 $, we denote $n^\varepsilon _i =\frac{\nabla u _i ^\varepsilon }{|\nabla u _i ^\varepsilon|}$.

We define the varifold corresponds to $\mu_t ^{i,\varepsilon}$ by
\[ V_t ^{i,\varepsilon} (\phi) :=\int_{\Omega\cap\{ |\nabla u _i ^\varepsilon |\not=0 \}} \phi (x, I-n^\varepsilon _i \otimes n^\varepsilon _i ) \, d\mu _t ^{i,\varepsilon},\quad \phi \in C_c(G_{d-1} (\Omega))  .\]
The first variation of $V_t ^{i,\varepsilon}$ is given by
\[ \delta V_t ^{i,\varepsilon} (g) = \int_{\Omega\cap\{ |\nabla u_i ^\varepsilon |\not=0 \}} ( I-n_i ^\varepsilon\otimes n _i ^\varepsilon ) :\nabla g \, d\mu _t ^{i,\varepsilon},\quad g \in C_c(\Omega;\mathbb{R}^d) . \]
\end{definition}
By the integration by parts(see \cite[Lemma 6.6]{takasaotonegawa}), we have
\begin{proposition}
For any $g \in C_c(\Omega;\mathbb{R}^d)$,
\begin{equation}
\begin{split}
\delta V_t ^{i, \varepsilon} (g) =& - \sigma^{-1}\int _{\Omega} (g\cdot \nabla u_i ^\varepsilon) \Big(-\varepsilon \Delta u_i ^\varepsilon +\frac{W'(u^\varepsilon _i )}{\varepsilon}\Big) \, dx\\
&+  \int_{\Omega\cap\{ |\nabla u _i ^\varepsilon |\not=0 \}}  (n _i ^\varepsilon \otimes n _i ^\varepsilon) : \nabla g \,d\xi _t ^{i,\varepsilon}
-\sigma^{-1}\int_{\Omega\cap\{ |\nabla u_i ^\varepsilon |=0 \}} I : \nabla g  \frac{W(u_i ^\varepsilon )}{\varepsilon} \, dx. 
\end{split}
\label{firstvariation}
\end{equation}
\end{proposition}

\begin{lemma}\label{lem3.5}
Assume that $\varepsilon \to 0$ is the positive sequence such that \eqref{eq3.3} holds. Then there exists a subsequence $\varepsilon \to 0$ (denoted by the same character) such that the following hold:
\begin{enumerate}
\item There exists $u_i \in BV _{loc}( \Omega \times [0,T)) \cap C^{\frac{1}{2}} _{loc} ([0,T) ;L^1 (\Omega))$ such that $u_i = 0$ or $1$ a.e. on $\Omega \times [0,T)$ and $u_i^{\varepsilon } \to u_i $ in  $L^1 _{loc} (\Omega \times [0,T)) $ and a.e. pointwise.
\item We have
\begin{equation}
\| \nabla u _i  (\cdot ,t)\| (\phi) \leq \mu _t ^i (\phi)
\label{eq3.4}
\end{equation}
for $i=1,2,\dots , N$, for any $\phi \in C_c (\Omega;[0,T))$, and for any $t \in [0,T)$. 
\item Assume \eqref{assumption}. Then we have
\begin{equation}
\mu _t ^i (\Omega) =\|\nabla u_i (\cdot,t) \| (\Omega)\quad \text{and} \quad \lim_{\varepsilon \to 0} \xi_t ^{i,\varepsilon} =0 \qquad \text{a.e.} \ t\in[0,T), \ i=1,2,\dots, N.
\label{eq3.5}
\end{equation}
\end{enumerate}
\end{lemma}

\begin{proof}
Step 1. Set $w^\varepsilon _i  := G \circ u _i ^\varepsilon $, where $G  (s) := \sigma ^ {-1} \int _{0 } ^s \sqrt{2W (y) } \, dy $. We remark that $G (-1) =0$ and $G (1) =1$. We compute that
\begin{equation}
|\nabla w_i ^\varepsilon |=\sigma ^{-1} |\nabla u_i ^\varepsilon | \sqrt{2W (u_i ^\varepsilon) } \leq \sigma ^{-1} \Big( \frac{\varepsilon |\nabla u_i ^\varepsilon |^2 }{2} +\frac{W (u_i ^\varepsilon)}{\varepsilon} \Big). 
\label{eq3.6}
\end{equation}
Therefore, by \eqref{eq2.4} and \eqref{eq3.1} we obtain
\begin{equation}
\int _{\Omega} |\nabla w_i ^\varepsilon (\cdot,t) | \, dx \leq \int _{\Omega} \sigma ^{-1} \Big( \frac{\varepsilon |\nabla u_i ^\varepsilon |^2 }{2} +\frac{W (\nabla u_i ^\varepsilon)}{\varepsilon} \Big) \, dx = \mu _t ^{i, \varepsilon} (\Omega) \leq D
\label{bv1}
\end{equation}
for $t \in [0,T)$. By the similar argument, \eqref{eq3.1}, and \eqref{eq3.6} we obtain
\begin{equation}
\begin{split}
\int _0 ^T \int_{\Omega} |\partial _t w_i ^\varepsilon | \, dxdt \leq \sigma ^{-1} \int _0 ^T \int_{\Omega} \Big( \frac{\varepsilon |\partial _t u _i ^\varepsilon |^2 }{2} +\frac{W (u _i ^\varepsilon)}{\varepsilon} \Big) \, dxdt \leq D(1+T).
\end{split}
\label{bv2}
\end{equation}
By \eqref{bv1} and \eqref{bv2}, $\{ w_i ^{\varepsilon_j} \}_{j=1} ^\infty$ is bounded in $BV  (\Omega \times [0,T))$. By the standard compactness theorem and the diagonal argument there is subsequence $\{ w_i ^{\varepsilon _j} \}_{j=1} ^\infty$ (denoted by the same index) and $w _i \in BV_{loc} (\Omega \times [0,T))$ such that
\begin{equation}
w^{\varepsilon  } _i \to w _i \quad \text{in} \ L^1 _{loc} (\Omega \times [0,T))
\label{wconv}
\end{equation}
and a.e. pointwise, for $ i=1,2, \dots, N$. We denote $u_i (x,t) := \lim _{\varepsilon \to 0} u^{\varepsilon } _i (x,t)=\lim _{\varepsilon \to 0} G^{-1} \circ w^{\varepsilon } _i (x,t)$. Then we have
\[ u_i^{\varepsilon } \to u_i \quad \text{in} \ L^1 _{loc} (\Omega \times [0,T)) \]
and a.e. pointwise, for $ i=1,2, \dots, N$. We have $u_i =1$ or $=0$ a.e. on $\Omega \times [0,T)$ by the boundedness of $\int _{\Omega} \frac{W (u_ i ^{\varepsilon}) }{\varepsilon_j} \, dx$. Moreover $u_i =w_i$ a.e. on $\Omega \times [0,T)$. Thus $u_i \in BV_{loc} (\Omega \times [0,T))$. For any open set $U \subset \Omega$ and a.e. $0\leq t_1 < t_2<T$ we have
\begin{equation}
\begin{split}
&\int_U |u_i (\cdot ,t_2) -u_i (\cdot ,t_1)| \, dx \leq \lim _{\varepsilon \to 0} \int _{\Omega} |w_i ^{\varepsilon } (\cdot ,t_2) -w_i ^{\varepsilon } (\cdot ,t_1)| \, dx \\
\leq & \liminf _{\varepsilon \to 0} \int_\Omega \int _{t_1} ^{t_2} |\partial _t w_i ^{\varepsilon }| \, dtdx \leq \liminf_{\varepsilon \to 0} \int_{\Omega} \int _{t_1} ^{t_2} \Big( \frac{\varepsilon |\partial _t u_i ^{\varepsilon } |^2 }{2}\sqrt{t_2 -t_1} +\frac{W (u_i ^{\varepsilon})}{\varepsilon \sqrt{t_2 -t_1}} \Big) \, dtdx \\
\leq & C D \sqrt{t_2 -t_1},
\end{split}
\label{eq3.10}
\end{equation}
where $C= C (d,T)>0$. By \eqref{eq3.10}, we may define $u_i (\cdot ,t)$ for any $t \in [0,T)$ such that $u_i \in C^{\frac{1}{2}} _{loc} ([0,T) ; L^1 (\Omega))$. For $\phi \in C_c (\Omega ;[0,T))$ and $t \in[0,T)$ we compute that
\begin{equation}
\begin{split}
&\int _{\Omega} \phi \, d\| \nabla u _i (\cdot , t) \| \leq \liminf _{\varepsilon \to 0} \int _{\Omega} \phi |\nabla w _i ^{\varepsilon  } | \, dx\\
\leq & \lim _{\varepsilon \to 0}\sigma ^{-1} \int _{\Omega} \phi \Big( \frac{\varepsilon |\nabla u_i ^{\varepsilon } |^2 }{2} +\frac{W (u_i ^{\varepsilon })}{\varepsilon} \Big) \, dx = \int _{\Omega} \phi \, d\mu _t ^i, \quad i=1,2,\dots, N.
\end{split}
\label{eq3.12}
\end{equation}
Therefore we obtain \eqref{eq3.4}.

\bigskip

\noindent
Step 2. By an argument similar to that in \eqref{eq3.12}, we have
\begin{equation}
\| \nabla u _i  (\cdot ,t)\| (\Omega) \leq \mu _t ^i (\Omega)
\label{eq3.12'}
\end{equation}
for $t \in [0,T)$ and $i=1,2,\dots, N$. Assume \eqref{assumption}. By \eqref{assumption} and \eqref{eq3.12'}, we have
\begin{equation}
\mu _t ^i (\Omega) =\|\nabla u_i (\cdot,t) \| (\Omega) \qquad \text{a.e.} \ t\in [0,T), \ i=1,2,\dots, N.
\label{eq3.11}
\end{equation}
Set $a^{\varepsilon} _i (x,t) = \sqrt{ \frac{\varepsilon |\nabla u_i ^{\varepsilon }(x,t) |^2 }{2} }$ and $b^{\varepsilon} _i (x,t) = \sqrt{ \frac{W (u_i ^{\varepsilon } (x,t) )}{\varepsilon } }$. Since $\phi_1 \equiv 1$ belongs to $C_c (\Omega)=C_c (\mathbb{T}^d)$, 
\begin{equation}
\mu _t ^{i,\varepsilon} (\Omega) =\mu _t ^{i,\varepsilon} (\phi_1)\to \mu _t ^{i} (\phi_1)=\mu _t ^{i} (\Omega) \qquad \text{for any} \ t\in [0,T), \ i=1,2,\dots, N. 
\label{eq3.11'}
\end{equation}
Fix $t\in[0,T)$ such that \eqref{eq3.11} holds. Then by \eqref{eq3.6}, \eqref{eq3.11} and \eqref{eq3.11'}, and \eqref{eq3.12} with $\phi \equiv 1$ we obtain
\begin{equation*}
\begin{split}
\limsup _{\varepsilon \to 0} \int_{\Omega}  (a^{\varepsilon} _i  - b^{\varepsilon} _i )^2 dx &= \limsup _{\varepsilon \to 0} \int_{\Omega} \{ (a^{\varepsilon} _i)^2  + (b^{\varepsilon} _i )^2  -2a^{\varepsilon} _i  b^{\varepsilon} _i\}dx\\
&= \limsup_{\varepsilon \to 0} \sigma \Big(\mu _t ^{i,\varepsilon } (\Omega ) - \int _{\Omega} |\nabla w_i ^\varepsilon | \, dx\Big)\\
&= \sigma \Big( \lim_{\varepsilon \to 0} \mu _t ^{i,\varepsilon } (\Omega) -\liminf_{\varepsilon \to 0} \int _{\Omega}  |\nabla w_i ^\varepsilon | \, dx\Big)\\
& \leq \sigma ( \mu _t ^{i} (\Omega) - \| \nabla u_i (\cdot ,t) \| (\Omega) )=0.
\end{split}
\end{equation*}
Thus we have
\begin{equation*}
\begin{split}
\limsup _{\varepsilon \to 0} |\xi _t ^{i,\varepsilon } | (\Omega )&= \sigma ^{-1} \limsup _{\varepsilon \to 0} \int_{\Omega} | (a^{\varepsilon} _i)^2  - (b^{\varepsilon} _i )^2| dx =\sigma ^{-1} \limsup _{\varepsilon \to 0} \int_{\Omega} |a^{\varepsilon} _i  + b^{\varepsilon} _i | \cdot |a^{\varepsilon} _i  - b^{\varepsilon} _i | dx \\
&\leq \sigma ^{-1} \limsup _{\varepsilon \to 0} \Big( \int_{\Omega}  (a^{\varepsilon} _i  + b^{\varepsilon} _i )^2 dx\Big)^{\frac12} \Big( \int_{\Omega}  (a^{\varepsilon} _i  - b^{\varepsilon} _i )^2 dx\Big) ^{\frac12}\\
&\leq \sigma ^{-1} C D^{\frac12} \limsup _{\varepsilon \to 0}  \Big( \int_{\Omega}  (a^{\varepsilon} _i  - b^{\varepsilon} _i )^2 dx\Big) ^{\frac12}=0.
\end{split}
\end{equation*}
Thus we obtain \eqref{eq3.5}.

\end{proof}

\begin{lemma}\label{lem3.6}
Assume \eqref{assumption}. Then there exists a family of varifolds $\{ V_t \}_{t\in [0,T)}$ such that for a.e. $t\in [0,T)$, $V_t \in \mathbf{IV}_{d-1} (\Omega)$, $\|V_t\| = \mu_t$, $V_t$ has a generalized mean curvature vector $h$, and $h$ satisfies
\begin{equation}
\int _{\Omega} |h| ^2 \, d \mu _t \leq \sigma ^{-1}\liminf _{\varepsilon\to 0} \int _{\Omega} \varepsilon |\partial _t u ^\varepsilon (x,t) |^2 \, dx<\infty
\label{eq3.19}
\end{equation}
and
\begin{equation}
\begin{split}
\delta V_t (g) &= -\int_{\Omega} h \cdot g \, d\mu _t \\
&= -\sigma^{-1}\lim_{\varepsilon \to 0} \sum _{i=1} ^N \int _{\Omega} (g\cdot \nabla u_i ^\varepsilon) \Big(-\varepsilon \Delta u_i ^\varepsilon +\frac{W'(u^\varepsilon _i )}{\varepsilon}\Big) \, dx\\
&= \sigma^{-1} \lim_{\varepsilon \to 0} \sum_{i=1} ^N \int _{\Omega} \varepsilon \partial _t u^\varepsilon _i \nabla u^\varepsilon _i \cdot g \, dx  , \quad g \in (C_c ^1 (\Omega))^d.
\label{eq3.13}
\end{split}
\end{equation}
Moreover $h \in L^2 (0,T ;L^2 (\mu_t ;\mathbb{R}^d))$.
\end{lemma}
\begin{proof} By \eqref{eq3.1} and Fatou's lemma, 
\begin{equation}
c(t):=\liminf_{\varepsilon \to 0}\int _{\Omega} \varepsilon |\partial _t u ^\varepsilon (x,t)|^2 \, dx <\infty
\label{eq3.14}
\end{equation}
for a.e. $t\in [0,T)$.
Assume \eqref{assumption}. Then \eqref{eq3.5} holds for a.e. $t \in [0,T)$. Fix $t\in [0,T)$ such that $\xi_t ^{i,\varepsilon} \to 0$ and \eqref{eq3.14} holds. By $\|V_t ^\varepsilon \| (\Omega)= \mu _t ^\varepsilon (\Omega) \leq D$ for any $\varepsilon >0$ and the compactness of the Radon measures, there exist $V_t \in \mathbf{V} _{d-1} (\Omega)$ and a subsequence $\varepsilon \to 0$ (denoted by the same character) such that $\|V_t \| =\mu _t$ and $V_t ^\varepsilon \to V_t$ as Radon measures. By \eqref{assumption} we have
\begin{equation}
\begin{split}
\mathcal{H}^{d-1} (\spt \|V_t \|) &= \mathcal{H}^{d-1} (\spt \mu_t ) =\mathcal{H}^{d-1} (\spt \|\nabla u (\cdot,t) \| ) \\
& \leq \sum_{i=1} ^N  \| \nabla u_i (\cdot, t) \|(\Omega) = \mu_t  (\Omega) <\infty.
\end{split}
\label{eq3.16}
\end{equation}
By \eqref{eq3.16} and a standard measure theoretic argument (see \cite[Proposition 6.1]{takasaotonegawa}) we have
\begin{equation}
\begin{split}
V_t = V_t \lfloor _{ \Big\{ x\in \Omega \, \Big| \, \limsup _{r\downarrow 0} \frac{\|V_t\| (B_r (x)) }{\omega _{d-1}r^{d-1}}  >0\Big\} \times G_{d-1} (\mathbb{R}^d) }.
\end{split}
\label{eq3.18}
\end{equation}

For any  $g \in (C_c(\Omega) )^d$, we compute that
\begin{equation*}
\begin{split}
 &\int_{\Omega\cap\{ |\nabla u _i ^\varepsilon |\not=0 \}}  (n _i ^\varepsilon \otimes n _i ^\varepsilon) : \nabla g \,d\xi _t ^{i,\varepsilon}
-\sigma^{-1}\int_{\Omega\cap\{ |\nabla u_i ^\varepsilon |=0 \}} I : \nabla g  \frac{W(u_i ^\varepsilon )}{\varepsilon} \, dx\\
=& \int_{\Omega\cap\{ |\nabla u _i ^\varepsilon |\not=0 \}}  (n _i ^\varepsilon \otimes n _i ^\varepsilon) : \nabla g \,d\xi _t ^{i,\varepsilon} +\int_{\Omega\cap\{ |\nabla u_i ^\varepsilon |=0 \}} I : \nabla g  \, d \xi _t ^{i,\varepsilon} \to 0.
\end{split}
\end{equation*}
By this, \eqref{eq3.2}, and \eqref{firstvariation}, we have
\begin{equation}
\begin{split}
\delta V _t (g) =& \int _{ G_{d-1} (\Omega)} S:\nabla g (x)  \, dV_t (x,S) =\lim _{\varepsilon \to 0}  \sum _{i=1} ^N \int_{\Omega\cap\{ |\nabla u_i ^\varepsilon |\not=0 \}} ( I-n_i ^\varepsilon\otimes n _i ^\varepsilon ) :\nabla g \, d\mu _t ^{i,\varepsilon}\\
=& -\sigma^{-1}\lim_{\varepsilon \to 0} \sum _{i=1} ^N \int _{\Omega} (g\cdot \nabla u_i ^\varepsilon) \Big(-\varepsilon \Delta u_i ^\varepsilon +\frac{W'(u^\varepsilon _i )}{\varepsilon}\Big) \, dx
\end{split}
\label{eq3.17}
\end{equation}  
for any $g \in (C_c ^1 (\Omega) )^d$. Moreover, by \eqref{eq3.14} and \eqref{eq3.17} we obtain
\begin{equation}
\begin{split}
| \delta V_t  (g)| = & \sigma^{-1}\lim_{\varepsilon \to 0} \Big| \sum _{i=1} ^N \int _{\Omega} (g\cdot \nabla u_i ^\varepsilon) \Big(-\varepsilon \Delta u_i ^\varepsilon +\frac{W'(u^\varepsilon _i )}{\varepsilon}\Big) \, dx\Big|\\
= & \sigma^{-1}\lim_{\varepsilon \to 0} \Big| \sum_{i=1} ^N \int _{\Omega} \varepsilon \partial _t u^\varepsilon _i \nabla u^\varepsilon _i \cdot g \, dx \Big|\\
\leq & C \| g\|_{C^0 (\Omega)} \liminf_{\varepsilon \to 0} \Big( \mu _t ^{\varepsilon} (\Omega)  \Big)^{\frac12} \Big(  \int _{\Omega} \varepsilon |\partial _t u ^\varepsilon|^2 \, dx  \Big)^{\frac12}  \\
\leq& CD^{\frac{1}{2}} c(t) \| g\|_{C^0 (\Omega)}
\end{split}
\label{eq3.15}
\end{equation}  
for any $g \in (C_c ^1 (\Omega) )^d$.  

By \eqref{eq3.18}, \eqref{eq3.15}, and Allard's rectifiability theorem (see \cite[Theorem 5.5]{allard}), $V_t \in \mathbf{RV} _{d-1} (\Omega) $.  Moreover we have $V_t \in \mathbf{IV} _{d-1} (\Omega)$ and $V_t$ is determined uniquely by $\|V_t\| =\mu_t$ and \eqref{assumption}.

By \eqref{eq3.15} and the Riesz representation theorem, $V_t$ has a generalized mean curvature vector $h$. By this, \eqref{eq3.2}, and \eqref{eq3.17} we have \eqref{eq3.13}. Let $\phi \in C_c ^2 (\Omega ;[0,\infty))$ with $\phi ^\frac12 \in C^1$. Then by an argument similar to that in \cite{Ilmanen} we have
\begin{equation}
\Big( \int_ \Omega \phi |h|^2 \, d \mu _t\Big) ^\frac12 =\sup \Big\{  \int_\Omega \phi ^\frac12 h\cdot g \, d\mu_t \, \Big| \, g \in C_c ^\infty (\Omega;\mathbb{R}^d), \ \|g \|_{L^2 (\mu_t)} \leq 1  \Big\}.
\label{eq3.20}
\end{equation}
For any $g\in C_c ^\infty (\Omega ; \mathbb{R}^d)$ we compute that
\begin{equation}
\begin{split}
 \int_\Omega \phi ^\frac12 h\cdot g \, d\mu_t &=-\delta V_t (\phi ^\frac12 g)=-\sigma^{-1} \lim_{\varepsilon \to 0} \sum_{i=1} ^N \int _{\Omega} \varepsilon \partial _t u^\varepsilon _i \nabla u^\varepsilon _i \cdot \phi^\frac12 g \, dx\\
&\leq \sigma^{-1} \liminf_{\varepsilon \to 0} \int _{\Omega} \varepsilon \Big( \sum _{i=1} ^N ( \phi ^{\frac{1}{2}} \partial _t u^\varepsilon _i) ^2 \Big)^{\frac{1}{2}} \Big( \sum_{i=1} ^N \Big( \sum_{j=1} ^d \partial _{x_j} u _i ^\varepsilon g ^j \Big)^2 \Big)^{\frac{1}{2}} \, dx\\
&\leq  \liminf_{\varepsilon \to 0}
\Big( \sigma ^{-1}\int _{\Omega} \varepsilon \phi |\partial _t u^\varepsilon  |^2 \, dx\Big)^\frac12 \Big( \sigma ^{-1} \int _{\Omega} |g|^2 \varepsilon |\nabla u^\varepsilon  |^2 \, dx\Big)^\frac12\\
&=  \liminf_{\varepsilon \to 0}
\Big( \sigma ^{-1}\int _{\Omega} \varepsilon \phi |\partial _t u^\varepsilon  |^2 \, dx\Big)^\frac12 \lim _{\varepsilon \to 0}\Big( \int _{\Omega} |g|^2  \, d \mu_t ^\varepsilon + \int _{\Omega} |g|^2  \, d \xi_t ^\varepsilon \Big)^\frac12\\
&= \liminf_{\varepsilon \to 0}
\Big( \sigma ^{-1}\int _{\Omega} \varepsilon \phi |\partial _t u^\varepsilon  |^2 \, dx\Big)^\frac12  \| g \|_{L^2 (\mu_t)},
\end{split}
\label{eq3.21}
\end{equation}
where \eqref{eq3.5} and \eqref{eq3.15} are used. Note that $\frac{\varepsilon}{\sigma} |\nabla u^\varepsilon |^2 dx = d\mu_t ^\varepsilon + d\xi_t ^\varepsilon$. By \eqref{eq3.20} and \eqref{eq3.21} we have $h \in L^2 (\mu_t ;\mathbb{R}^d)$ and \eqref{eq3.19}. Similarly, we have $h \in L^2 (0,T ;L^2 (\mu_t ;\mathbb{R}^d))$ by \eqref{eq3.1}.

For $t \in [0,T)$ such that \eqref{eq3.5} or \eqref{eq3.14} does not hold, we define  $V_t$ by $V_t  (\phi) :=\int_{\Omega} \phi (x, P_1) \, d\mu _t $ for $ \phi \in C_c(G_{d-1} (\Omega))$. Here $P_1=\{ x \in \mathbb{R}^d \, | \, x_d =0 \}$.
\end{proof}

\begin{proof}[Proof of Theorem \ref{mainresults}] From Lemma \ref{lem3.2}, Lemma \ref{lem3.5}, and Lemma \ref{lem3.6}, we only need to prove (c) and (d). First we show Brakke's inequality (c). Let $\phi \in C _c ^2 (\Omega \times [0,T);\mathbb{R}^+)$ and $0\leq t_1 < t_2 <T$. By the integration by parts, we have
\begin{equation}
\begin{split}
\frac{d}{dt} \int _{\Omega} \phi \, d\mu_t ^\varepsilon 
=& \int _\Omega \partial _t \phi \, d\mu _t ^\varepsilon 
+\sigma ^{-1}\sum _{i=1} ^N \int _{\Omega} \Big( \varepsilon \phi ( - \partial _t u_i ^\varepsilon +\lambda ^\varepsilon u_i ^\varepsilon )\partial _t u_i ^\varepsilon - \varepsilon (\nabla \phi \cdot \nabla u_i ^\varepsilon ) \partial_t u_i ^\varepsilon \Big) \, dx\\
=& \int _\Omega \partial _t \phi  \, d\mu _t ^\varepsilon 
+\sigma ^{-1}\int _{\Omega} \Big(-\varepsilon \phi |\partial _t u ^\varepsilon|^2 - \sum _{i=1} ^N  \varepsilon (\nabla \phi \cdot \nabla u_i ^\varepsilon ) \partial_t u_i ^\varepsilon \Big)\, dx.
\end{split} 
\end{equation}
Here \eqref{eq1.2} is used. By \eqref{eq3.3}, we have
\begin{equation*}
\lim _{j \to \infty } \int _{\Omega} \phi \, d \mu_t ^{\varepsilon_j} \Big| _{t=t_1} ^{t_2} = \int _{\Omega} \phi \, d\mu_t \Big| _{t=t_1} ^{t_2} \quad \text{and} \quad \lim _{j \to \infty } \int _{t_1} ^{t_2} \int _{\Omega} \partial _t \phi \, d \mu_t ^{\varepsilon_j} dt = \int_{t_1} ^{t_2} \int _{\Omega} \partial_ t \phi \, d\mu_t dt.
\end{equation*}
Thus we only need to show
\begin{equation}
\int_{t_1} ^{t_2} \int _{\Omega} (\phi |h|^2 - \nabla \phi \cdot h) \, d\mu_t dt
\leq
\sigma ^{-1}\lim _{j\to \infty} \int_{t_1} ^{t_2} \int _{\Omega} \Big( \varepsilon_j \phi |\partial _t u ^{\varepsilon_j} |^2 + \sum _{i=1} ^N  \varepsilon (\nabla \phi \cdot \nabla u_i ^{\varepsilon_j} ) \partial _t u_i ^{\varepsilon_j} \Big)\, dxdt .
\label{eq3.24}
\end{equation}
There exists $C=C(\|\eta \|_{C^2 (\Omega)}) >0$ such that $\sup _{\Omega}\frac{|\nabla \eta |^2}{\eta} \leq C$, for any $\eta \in C_c ^2 (\Omega;\mathbb{R}^+)$. Therefore we have
\begin{equation}
\begin{split}
& \int _{\Omega} \Big( \varepsilon_j \phi |\partial _t u ^{\varepsilon_j}|^2 + \sum _{i=1} ^N  \varepsilon_j (\nabla \phi \cdot \nabla u_i ^{\varepsilon_j} ) \partial _t u_i ^{\varepsilon_j} \Big)\, dx \\
\geq & \sum _{i=1} ^N \int _{\Omega} \Big( \varepsilon_j \phi \Big( \partial _t u_i ^{\varepsilon_j} + \frac{\nabla \phi \cdot \nabla u_i ^{\varepsilon_j}}{2\phi} \Big) ^2  -  \frac{\varepsilon _j (\nabla \phi \cdot \nabla u_i ^{\varepsilon_j})^2}{4\phi} \Big)\, dx\\
\geq & -C(\|\phi \|_{C^2 (\Omega)}) D
\end{split}
\label{eq3.25}
\end{equation}
for any $t \in [0,T)$ and $\varepsilon \in (0,1)$. Fatou's Lemma and \eqref{eq3.25} imply
\begin{equation}
\begin{split}
& \int_{t_1} ^{t_2} \liminf _{j\to \infty}  \int _{\Omega} \Big( \varepsilon_j \phi |\partial _t u ^{\varepsilon_j} |^2 + \sum _{i=1} ^N  \varepsilon (\nabla \phi \cdot \nabla u_i ^{\varepsilon_j} ) \partial _t u_i ^{\varepsilon_j} \Big)\, dxdt \\
\leq &
\lim _{j\to \infty} \int_{t_1} ^{t_2} \int _{\Omega} \Big( \varepsilon_j \phi |\partial _t u ^{\varepsilon_j} |^2 + \sum _{i=1} ^N  \varepsilon (\nabla \phi \cdot \nabla u_i ^{\varepsilon_j} ) \partial_t u_i ^{\varepsilon_j} \Big)\, dxdt .
\end{split}
\label{eq3.26}
\end{equation}
By \eqref{eq3.5} and \eqref{eq3.26} for a.e. $t\in [0,T)$ we have
\begin{equation}
\lim _{j\to \infty}\xi _t ^{\varepsilon_j} = 0 \quad \text{and} \quad \liminf _{j\to \infty}  \int _{\Omega} \Big( \varepsilon_j \phi |\partial _t u ^{\varepsilon_j} |^2 + \sum _{i=1} ^N  \varepsilon (\nabla \phi \cdot \nabla u_i ^{\varepsilon_j} ) \partial_t u_i ^{\varepsilon_j} \Big)\, dx <\infty.
\label{eq3.27}
\end{equation}
We fix such $t\in [0,T)$. We compute that
\begin{equation}
\varepsilon_j \phi |\partial _t u ^{\varepsilon_j} |^2 + \sum _{i=1} ^N  \varepsilon (\nabla \phi \cdot \nabla u_i ^{\varepsilon_j} ) \partial _t u_i ^{\varepsilon_j} \geq \frac{1}{2} \varepsilon_j \phi |\partial _t u ^{\varepsilon_j} |^2 -C(\|\phi \|_{C^2 (\Omega)}) D.
\label{eq3.28}
\end{equation}
By \eqref{eq3.27} and \eqref{eq3.28} we have
\begin{equation}
\liminf _{j\to \infty}  \int _{\Omega} \varepsilon_j \phi |\partial _t u ^{\varepsilon_j} |^2 \, dx <\infty.
\label{eq3.29}
\end{equation}
Hence, by an argument similar to that in Lemma \ref{lem3.6}, $V_t ^{\varepsilon _j}\lfloor _{\{ \phi >0 \}} \to V_t \lfloor _{\{ \phi >0 \}} $,
\begin{equation}
\int _{\Omega} \phi |h| ^2 \, d \mu _t \leq \sigma ^{-1} \liminf _{j\to \infty}  \int _{\Omega} \varepsilon_j \phi |\partial _t u ^{\varepsilon_j} |^2 \, dx
\label{eq3.30}
\end{equation}
and \eqref{eq3.13} holds for any $g \in (C_c ^1 (\{ \phi >0\}))^d$. Thus
\begin{equation}
-\int _{\Omega} \nabla \tilde \phi \cdot h \, d\mu_t = 
\sigma ^{-1}\lim _{j\to \infty} \int _{\Omega} \sum _{i=1} ^N  \varepsilon (\nabla \tilde \phi \cdot \nabla u_i ^{\varepsilon_j} ) \partial u_i ^{\varepsilon_j} \, dx, \quad  \tilde \phi \in (C_c ^2 (\{ \phi >0\}))^d.
\label{eq3.31}
\end{equation}
By \eqref{eq3.30} and \eqref{eq3.31} with $\tilde \phi \to \phi $ in $C^2$, we obtain \eqref{eq3.24}. Therefore we have (c).

Finally we show (d). In the same manner as \cite[3.3. Monotonicity formula]{Ilmanen} and \cite[Proposition 4.1]{takasaotonegawa}, we have
\begin{equation}
\begin{split}
&\frac{d}{dt} \int _{\mathbb{R}^d} \rho \, d\mu _t ^{i, \varepsilon}(x) \\
=& -\int _{\mathbb{R}^d} \varepsilon \rho \Big( -\Delta u_i ^{\varepsilon} +\frac{W'  (u_i ^{\varepsilon}) }{\varepsilon ^2} -\frac{\nabla u_i ^{\varepsilon} \cdot \nabla \rho}{\rho} \Big)^2 \, d\mu _t^{i, \varepsilon}(x) + \frac{1}{2(s-t)}\int _{\mathbb{R}^d}  \rho \, d\xi _t ^{i, \varepsilon}(x) \\
&+ \frac{1}{2\sigma}\int _{\mathbb{R}^d} \rho \lambda ^\varepsilon \partial _t  (u_i ^\varepsilon )^2 \, dx + \frac{1}{2\sigma}\int _{\mathbb{R}^d} \lambda ^\varepsilon \nabla \rho \cdot \nabla  (u_i ^\varepsilon )^2 \, dx
\end{split}
\label{monoton2}
\end{equation}
for $y\in \mathbb{R}^d$, $t<s$ with $t \in [0,T)$, and $i=1,2,\dots, N$. 
Note that
\begin{equation}
\sum_{i=1} ^N \Big( \frac{1}{2\sigma}\int _{\mathbb{R}^d} \rho \lambda ^\varepsilon \partial _t  (u_i ^\varepsilon )^2 \, dx + \frac{1}{2\sigma}\int _{\mathbb{R}^d} \lambda ^\varepsilon \nabla \rho \cdot \nabla  (u_i ^\varepsilon )^2 \, dx\Big)=0
\label{eq3.34}
\end{equation}
by \eqref{eq1.2}. Thus \eqref{eq3.5}, \eqref{monoton2}, and \eqref{eq3.34} imply \eqref{monoton}.

\end{proof}

\section{Final remarks}
In this section, we consider the following systems of the Allen-Cahn equations:
\begin{equation}
\left\{ 
\begin{array}{ll}
\varepsilon \partial _t v^{\varepsilon}_i  =\varepsilon \Delta v_i ^\varepsilon - \dfrac{W ' (v_i ^\varepsilon)}{\varepsilon} + \Lambda ^\varepsilon _1 ,& (x,t)\in \Omega \times (0,T),  \\
v^{\varepsilon} _i (x,0) = v^{\varepsilon} _{0,i} (x) ,  &x\in \Omega,
\end{array} \right.
\label{ac5}
\end{equation} 
and
\begin{equation}
\left\{ 
\begin{array}{ll}
\varepsilon \partial _t w^{\varepsilon}_i  =\varepsilon \Delta w_i ^\varepsilon - \dfrac{W ' (w_i ^\varepsilon)}{\varepsilon} + \Lambda ^\varepsilon _2 \sqrt{ 2W (w_i ^\varepsilon)} ,& (x,t)\in \Omega \times (0,T),  \\
w^{\varepsilon} _i (x,0) = w^{\varepsilon} _{0,i} (x) ,  &x\in \Omega,
\end{array} \right.
\label{ac4}
\end{equation} 
where 
\begin{equation*}
 \Lambda^\varepsilon _1 =  \frac{1}{N}\sum _{i=1} ^N \Big(- \varepsilon \Delta v_i ^\varepsilon + \dfrac{W ' (v_i ^\varepsilon)}{\varepsilon} \Big) \ \ \text{and} \ \ \Lambda^\varepsilon _2 = \frac{ \sum _{i=1} ^N \sqrt{2W(w_i ^\varepsilon)} \Big(- \varepsilon \Delta w_i ^\varepsilon + \dfrac{W ' (w_i ^\varepsilon)}{\varepsilon} \Big) }{ \sum_{i=1} ^N 2W(w_i ^\varepsilon) }.
\label{lambda2}
\end{equation*}
The solutions for \eqref{ac5} and \eqref{ac4} with $\sum _{i=1} ^N v_{0,i} ^\varepsilon (x) \equiv 1$ and $\sum _{i=1} ^N k (w_{0,i} ^\varepsilon (x))\equiv \frac{1}{6}$ satisfy
\[\frac{\partial }{\partial t} \sum_{i=1} ^N v_i ^\varepsilon =0 \quad \text{and} \quad \frac{\partial }{\partial t}\sum_{i=1} ^N k(w^\varepsilon _i)=0,\]
where $k(s) = \int _0 ^s \sqrt{2W(a)} \, da$. Thus
\begin{equation}
\sum _{i=1} ^N v_i ^\varepsilon (x,t) =1 \quad \text{and} \quad \sum _{i=1} ^N k( w_i ^\varepsilon (x,t) )=\frac{1}{6} \qquad \text{for any} \ (x,t) \in \Omega \times (0,T).
\label{eq4.2}
\end{equation}
Note that $k(0)=0$ and $k(1) =\frac{1}{6}$ and 
\[ \lim _{\varepsilon \to 0} v_i ^\varepsilon (x,t) = 0 \ \text{or} \ 1 \ \ \text{and} \ \ \lim _{\varepsilon \to 0} w_i ^\varepsilon (x,t) = 0 \ \text{or} \ 1 \ \text{a.e.} \ (x,t) \in \Omega \times [0,T)\] 
by an argument similar to that in Lemma \ref{lem3.5}. Thus \eqref{eq4.2} corresponds to \eqref{eq1.2} and \eqref{eq1.7}. The solution $v_i ^\varepsilon$ satisfies Proposition \ref{prop3.1} by substituting $v_i ^\varepsilon$ into $u_i ^\varepsilon$ (the same is true of  $w_i ^\varepsilon$). Therefore we have
\begin{theorem}
Let $d,N\geq 2$. Assume that $v^\varepsilon$ (or $w^\varepsilon$) is a classical solution for \eqref{ac5} (resp. \eqref{ac4}) with \eqref{eq4.2}, and the initial data satisfies \eqref{eq2.4}, by substituting $v_i ^\varepsilon$ (resp. $w_i ^\varepsilon$) into $u_i ^\varepsilon$, for the definition of $\mu _t ^\varepsilon$. Then, under the Assumption A there exists a subsequence $\varepsilon \to 0$ such that Theorem \ref{mainresults} holds with $u_i ^\varepsilon = v^\varepsilon _i$ (resp. $u_i ^\varepsilon = w^\varepsilon _i$).
\end{theorem}

%

\noindent
\textbf{Acknowledgments}

\noindent
This work was supported by JSPS KAKENHI
Grant Numbers JP16K17622, JP17J02386, 
and JSPS Leading Initiative for Excellent Young Researchers(LEADER) operated by Funds for the Development of Human Resources in Science and Technology.


\bigskip
\noindent
{\small 
\qquad Keisuke Takasao\\
\qquad Department of Mathematics/Hakubi Center,\\
\qquad Kyoto University\\
\qquad Kitashirakawa-Oiwakecho, Sakyo,\\
\qquad Kyoto 606-8502, Japan\\
\qquad E-mail address: k.takasao@math.kyoto-u.ac.jp}
\end{document}